\documentclass[12pt]{article}

\usepackage[letterpaper,margin=1in]{geometry}

\newcommand{\triplenorm}[1]{
  \left\vert\kern-0.9pt\left\vert\kern-0.9pt\left\vert #1
  \right\vert\kern-0.9pt\right\vert\kern-0.9pt\right\vert}

\usepackage{amsmath}
\usepackage{amsbsy}
\usepackage{amsfonts}
\usepackage{amssymb}
\usepackage{amscd}
\usepackage{mathrsfs}

\usepackage{bm}

\usepackage{algpseudocode}
\usepackage[Algorithm]{algorithm}

\algnewcommand{\IIf}[1]{\State\algorithmicif\ #1\ \algorithmicthen}
\algnewcommand{\EElse}{\unskip\ \algorithmicelse\ }
\algnewcommand{\EndIIf}{\unskip\ \algorithmicend\ \algorithmicif}
\algnewcommand{\FFor}[1]{\State\algorithmicfor\ #1\ }
\algnewcommand{\EndFFor}{\unskip\ \algorithmicend\ \algorithmicfor}

\usepackage{txfonts}

\usepackage[latin1]{inputenc}

\usepackage{graphicx}
\usepackage{subfigure} 
\usepackage{float}

\usepackage{tikz}   
\usetikzlibrary{shapes,arrows,decorations.pathmorphing,backgrounds,positioning,fit,matrix,calc}

\tikzstyle{decision} = [diamond, draw, fill=blue!20, 
    text width=4.5em, text badly centered, node distance=3cm, inner sep=0pt]
\tikzstyle{block} = [rectangle, draw, fill=blue!20, 
    text width=5em, text centered, rounded corners, minimum height=4em]
\tikzstyle{line} = [draw, -latex']
\tikzstyle{cloud} = [draw, ellipse,fill=red!20, node distance=3cm,
    minimum height=2em]

\usepackage{hyperref}
\usepackage{type1cm}       

\usepackage{soul}

\usepackage{url}

\usepackage{rotating} 
\usepackage{makeidx}
\makeindex             
\usepackage{multicol}   
\usepackage{enumerate}
\usepackage{xspace}
 \usepackage{comment}  
\usepackage{etoolbox}

\usepackage{fancyhdr}
\ifcsmacro{theorem}{}{
\newtheorem{theorem}{Theorem}[section]
}
\ifcsmacro{lemma}{}{
\newtheorem{lemma}[theorem]{Lemma}
}
\ifcsmacro{corollary}{}{

}
\ifcsmacro{proposition}{}{

}
\ifcsmacro{algorithm}{}{

}

\let\oldchapter\chapter
\def\chapter{
  \setcounter{exercise}{0}
  \oldchapter
}

\newcommand{\beas}{\begin{eqnarray*}}
\newcommand{\eeas}{\end{eqnarray*}}
\newcommand{\bary}{\begin{array}}
\newcommand{\eary}{\end{array}}

\def\ec{\mathrel{\hbox{$\copy\Ea\kern-\wd\Ea\raise-3.5pt\hbox{$\sim$}$}}}
\newcommand{\lc}{\mathrel{\raise2pt\hbox{${\mathop<\limits_{\raise1pt\hbox{\mbox{$\sim$}}}}$}}}
\newcommand{\gc}{\mathrel{\raise2pt\hbox{${\mathop>\limits_{\raise1pt\hbox{\mbox{$\sim$}}}}$}}}

\newbox\Ea
\setbox\Ea=\hbox{\raise0.9pt\hbox{$=$}}

\newcommand{\bproof}{\begin{proof}}
\newcommand{\eproof}{\end{proof}}

\ifcsmacro{R}{}{
 
}

\newtheorem{remark}[theorem]{Remark}
\newtheorem{definition}[theorem]{Definition}

\begin{document}

\title{On the Weyl's law for discretized elliptic operators}
\author{Jinchao Xu, Hongxuan Zhang and Ludmil Zikatanov}

\maketitle

\begin{abstract}
  In this paper we give an estimate on the asymptotic behavior of
  eigenvalues of discretized elliptic boundary values problems. We
  first prove a simple min-max principle for selfadjoint operators on
  a Hilbert space. Then we show two sided bounds on the $k$-th
  eigenvalue of the discrete Laplacian by the $k$-th eigenvalue of the
  continuous Laplacian operator under the assumption that the finite
  element mesh is quasi-uniform. Combining this result with the
  well-known Weyl's law, we show that the $k$-th eigenvalue of the
  discretized isotropic elliptic operators, spectrally equivalent to
  the discretized Laplacian, is $\mathcal O\left(k^{2/d}\right)$.
  Finally, we show how these results can be used to obtain an error
  estimate for finite element approximations of elliptic eigenvalue
  problems.
\end{abstract}

\section{Introduction}

In this work, we focus on the finite element method as discretization of
solving linear partial differential equations of the form
\begin{equation}\label{e1}
    \mathcal L u= f,
\end{equation}
with given boundary conditions.  As is customary in finite element
method, the solution to~\eqref{e1} is approximated by a piece-wise
polynomial function (see~\cite{2002CiarletP-aa}). The behavior of the
eigenvalues of the linear operator $\mathcal L$ on the corresponding
finite element space when the mesh size becomes small are instrumental
in showing optimality of the methods, estimating the error of
approximation, construction of efficient solvers for the resulting
linear systems, and in all other components of the finite element
analysis. Classical references for the finite
element approximation of compact eigenvalue problems are the works by Strang and Fix~\cite{strang1973analysis}
and later the so-called
Babu\v{s}ka-Osborn
theory\cite{babuvska1987estimates,babuvska1989finite,babuska1991eigenvalue}. The
main result of this theory can be summarized as:
\begin{equation}\label{babuska-osborn}
    \lambda_k\le \lambda_{h,k}\le \lambda_k + C_k\sup_{\substack{u\in E_k\\\|u\|_V=1}}\inf_{v\in V_h}\|u-v\|_V^2,
\end{equation}
where $C_k$ are a constants depending on $k$ and $E_k$ denotes the
eigenspace associated with $\lambda_k$, the $k$-th eigenvalue of the
continuous Laplacian.  The proof of \eqref{babuska-osborn} found in
\cite{babuvska1987estimates} relies on an induction argument, and does
not seem to provide bounds on $C_k$ which are independent of $k$. Some
refinements on the estimates of the constants $C_k$ in
\eqref{babuska-osborn} are given in~\cite{babuvska1989finite,babuska1991eigenvalue}, and, more recently in~\cite{lin2014lower,zhang2015many},
and also follow from the main result in the present paper. We also mention that
bounding the eigenvalues of the discrete operators plays a crucial
role in the convergence analysis of multigrid
methods~\cite{2008ZikatanovL-aa,Falgout.R;Vassilevski.P.2004a,Falgout.R;Vasilevski.P;Zikatanov.L.2005a}.

In \cite{weyl1911asymptotische}, Hermann Weyl proved the following
asymptotic formula for the eigenvalues of the Dirichlet Laplacian in a
bounded domain $\Omega\subset \mathbb R^d$:
\[
    \lim_{k\rightarrow\infty}\frac{\lambda_k}{k^{\frac{2}{d}}}=\frac{(2\pi)^2}{[\omega_d \operatorname{Vol}(\Omega)]^{\frac{2}{d}}},
\]
where $\omega_d$ is a volume of the unit ball in $\mathbb{R}^d$, and
$\operatorname{Vol}(\Omega)$ is the volume of $\Omega$. This formula
was actually conjectured independently by Arnold Sommerfeld
\cite{sommerfeld1910greensche} and Hendrik Lorentz
\cite{lorentz1910alte} in 1910 who stated Weyl's law as a conjecture
based on the book of Lord Rayleigh
\cite{rayleigh1896theory}. Recently, V. Ivrii~\cite{2016IvriiV-aa}.
provided comprehensive summary on the research on Weyl's law since its
discovery.

In this paper the main result is a sharp estimate on the asymptotic
behavior of eigenvalues for Laplacian operator on the finite element
spaces as follows
\begin{equation}
    \lambda_k\le \lambda_{h,k}\le C\lambda_k,
\end{equation}
with $C$ independent of $k$, $h$.  We show such an estimate under the
assumption that the finite element mesh is quasi-uniform and in
combination with Weyl's Law for the PDE this result shows that the
$k$-th smallest eigenvalue of the discretized Laplacian operator is
$\mathcal O(k^{2/d})$. This leads to the conclusion that the
eigenvalues of the finite element discrete operator exhibit the same
asymptotic behavior as the eigenvalues of the continuous Laplacian
operator.

This paper is organized as follows. In Section \ref{s:model} we
introduce the basic model problem and some related results on the
stability and approximation property of $L^2$-projection onto finite
element space. In Section \ref{s:min-max}, we show a generalized
min-max principle for selfadjoint operators on separable Hilbert
space. This is an important tool, used in the proof of our main
result. In Section \ref{s:algebraic-spectral}, we recall Weyl's law
for the Laplacian operator on Sobolev spaces and prove the main
eigenvalue asymptotic estimate for the discretized Laplacian operator.
In Section \ref{s:error}, we prove an error estimate for finite
element approximation of 2nd-order elliptic eigenvalue problems.

\section{Model elliptic PDE operators and finite element discretization}\label{s:model}
We consider the following boundary value problems
\begin{equation}
  \label{Model0}
    {\mathcal L}u=-\nabla\cdot (\alpha(x)\nabla u)+ q(x) u=f, \quad x\in \Omega
\end{equation}
where $\alpha: \Omega\mapsto \mathbb R_{{\rm sym}}^{d\times d}$ is a
matrix valued function taking values in the set of $d\times d$,
symmetric, positive definite matrices.
\begin{equation}
  \label{alpha}
\alpha_0\|\xi\|^2\le
\xi^T\alpha (x)\xi \le
\alpha_1\|\xi\|^2,\quad\mbox{for all}\quad \xi\in \mathbb R^d,
\end{equation}
for some positive constants $\alpha_0$ and $\alpha_1$ and all $x\in
\Omega$.  The potential $q(x)$ is assumed to be bounded and
non-negative for almost all $x\in \Omega$. Here, $d=1,2,3$ and
$\Omega\subset\mathbb R^d$ is a bounded domain with sufficiently
smooth (Lipschitz) boundary $\Gamma=\partial \Omega$.

The variational formulation of \eqref{Model0} is: Find
$u\in V$ such that
\begin{equation}
  \label{Vari}
a(u,v)=(f, v), \quad\forall v\in V.
\end{equation}
where the bilinear form $a(\cdot,\cdot)$ and the linear form
$(f,\cdot)$ are defined as
\begin{eqnarray*}
&& a(u,v)=\int_\Omega (\alpha(x)\nabla u)\cdot \nabla v
 + \int_{\Omega} q\;uv,\\
&& (f,v)= \int_\Omega fv.
\end{eqnarray*}
The Sobolev space $V$ that can be chosen according to the
boundary conditions accompanying the equation~\eqref{Model0}.
For example, in the case of mixed boundary conditions:
\begin{equation}
  \label{MixedBoundary}
  \begin{array}{rcl}
u=&0, &x\in \Gamma_D,\\
(\alpha\nabla u)\cdot n=&0,&x \in\Gamma_N,
\end{array}
\end{equation}
where
$\Gamma=\Gamma_D\cup\Gamma_N$.  The pure Dirichlet problem is when
$\Gamma_D = \Gamma$ while the pure Neumann problem is when $\Gamma_N
=\Gamma$. We then have $V$ defined as
\begin{equation}
  \label{3V}
V=\left\{
  \begin{array}{l}
    H^1(\Omega) = \{v\in L^2(\Omega): \partial_iv\in    L^2(\Omega), i=1:d\};\\
    H^1_D(\Omega) = \{v\in H^1(\Omega): v|_{\Gamma_D}=0\}.
  \end{array}
\right.
\end{equation}
When we consider a pure Dirichlet problem, $\Gamma_D = \Gamma$, we
denote the space by $V=H^1_0(\Omega)$. In addition, for pure Neumann
boundary conditions,  when $q=0$ the condition
\(\label{consistent-f} \int_\Omega f =0\) is usually added to assure
uniqueness of the solution to~\eqref{Vari}. We note that $a(.,.)$ defines and inner product on $V$ and 
\[
|\cdot|^2_a:=a(u,u),
\]
is a norm (or a seminorm) on $V$.

One most commonly used model problem is when $\alpha(x)=I$, and
$q(x)=0$ for all $x\in \Omega$, which corresponds to the
Poisson equation
\begin{equation}
  \label{Poisson}
-\Delta u=f.
\end{equation}
This simple problem provides a good representative model for isotropic
problems. In the following discussion, we assume
$\alpha(x) = I$  and $q(x)=0$. We further remark
that the results carry over to other isotropic elliptic
equations by spectral equivalence.

Given a triangulation ${\mathcal T}_h$ of $\Omega$, let $V_h\subset V$
be finite element space consisting of piecewise linear (or higher
order) polynomials with respect to this triangulation
${\mathcal T}_h$.  By triangulation here we mean $d$-homogenous
simplicial complex in $\mathbb{R}^d$ which covers $\Omega$.  The
finite element approximation of the variational problem \eqref{Vari} then
is: Find $u_h\in V_h$ such that
\begin{equation}
  \label{vph}
a(u_h,  v_h)=(f,v_h), \quad\forall\,v_h\in V_h.
\end{equation}
Given a basis $\{\phi_i\}_{i=1}^{N}$ in $V_h$,  we write
\(u_h(x)=\sum_{j=1}^{N}\mu_j\phi_j(x) \)
the equation \eqref{vph}  is then equivalent to
\[
        \sum_{j=1}^{N}\mu_ja(\phi_j,\phi_i)=(f,\phi_i),\quad
        j=1,2,\cdots, N,
\]
which is a linear system of equations:
\begin{equation}\label{axb}
        A\mu=b, \quad (A)_{ij} = a(\phi_j,\phi_i), \quad
 \mbox{and}\quad (b)_i=(f,\phi_i).
\end{equation}
Here, the matrix $A$ is known as the stiffness matrix of the  nodal basis
$\{\phi_i  \}_{i=1}^N$.

With any simplex $T\in \mathcal T_h$, we associate  the following geometric characteristics:
\begin{equation}\label{hT}
  \begin{array}{l}
    \displaystyle \overline h_T=\mbox{\rm diam\,}(T),\quad h_T=|T|^{\frac{1}{d}}, \quad  h=\max_{T\in \mathcal T_h} \overline h_T,
\\
    \displaystyle
    \underline h_T=2\sup\{r>0: B(x, r)\subset T \text{ for } x\in T\}.
    \end{array}
\end{equation}
In the following discussion, we need the definition of
\emph{quasi-uniform} finite element mesh and we recall this next.
\begin{definition}[Quasi-uniform mesh]
  We say that the mesh $\mathcal T_h$ is \emph{quasi-uniform} if there
  exists a constant $C > 0$ such that
    \begin{equation}
        \max_{T\in \mathcal T_h} \frac{h}{\underline h_T} \le C.
    \end{equation}
\end{definition}

We assume that we have a projection $\Pi_h: V\mapsto V_h$, satisfying
\begin{equation}\label{stability}
            |\Pi_hv|_{a}^2 \le c_1|v|_{a}^2, \quad \forall v\in V,
\end{equation}
and
\begin{equation}\label{approximation}
\|v-\Pi_hv\|_0^2 \le c_2h^2|v|_{a}^2,
          \quad \forall v\in V,
        \end{equation}
with $c_1$ and $c_2$ being constants independent of $h$ and $v$.

One example of $\Pi_h$ is the $L^2$ projection $Q_h:V\mapsto V_h$ defined by
\begin{equation}\label{Qh}
(Q_hv, w)_{L^2(\Omega)} = (v, w)_{L^2(\Omega)}, \quad \forall v\in V, w\in V_h.
\end{equation}
and a proof that $Q_h$ satisfies \eqref{stability} and \eqref{approximation} can be found in~\cite{1992XuJ-aa}, \cite{1991BrambleJ_XuJ-aa}, and \cite{2014BankR_YserentantH-aa}.

Another example of $\Pi_h$ is the elliptic projection $P_h: V\mapsto V_h$ defined as
\[\label{Ph}
    a(P_hv, w) = a(v, w), \quad \forall v\in V, w\in V_h.
\]
It is well-known that the elliptic projection
$P_h$ satisfies \eqref{stability} with $c_1=1$,
and, moreover, it also satisfies
\eqref{approximation} if the weak
solution to the differential equation~\eqref{Model0}
is $H^2$-regular, namely, there exists a constant $C_r$ such that
\begin{equation}
    |u|_2^2 \le C_r\|g\|_0^2,
\end{equation}
where $u$ is the solution of
\[
    a(w, u) = (f, w), \quad \forall w\in V.
\]

Next section discusses the min-max principle in (finite dimensional) Hilbert space
\section{On the min-max principle}\label{s:min-max}

The well known min-max principle for eigenvalues of symmetric matrices
was probably first stated and proved in~\cite{courant1924methoden}. We
state it as the following theorem.
\begin{theorem}[Min-max principle A]\label{thm:minmax}
  Let $A$ be a $n\times n$ symmetric matrix with eigenvalues
  $\lambda_{1}\le \lambda_2\le \cdots \le \lambda_n$, then
    \begin{equation}\label{minmax}
        \lambda_k=\min_{\dim W=k}\max_{x\in W, x\ne 0}\frac{(Ax, x)}{(x, x)},
    \end{equation}
    and
    \begin{equation}\label{maxmin}
        \lambda_k=\max_{\dim W=n-k+1}\min_{x\in W, x\ne 0}\frac{(Ax, x)}{(x, x)}.
    \end{equation}
\end{theorem}

We now consider the case when $A: X \mapsto V$ is a selfadjoint
operator in a Hilbert space $V$. We assume that the domain of $A$,
is $X \subset V$ and $X$ is dense in $V$.

The following lemma is used in the proof of the min-max principle. The result seems obvious (and is obvious in finite dimensional space).
\begin{lemma}\label{l:intersection} Let $V$ be a separable Hilbert space
  with orthonormal
  basis $\{\varphi_j\}_{j=1}^{\infty}$ and let for a fixed integer $k\ge 1$,
\begin{equation}
V_{k+} = \operatorname{span}\{\varphi_j\}_{j=k}^\infty =
\operatorname{span}\{\varphi_k,\varphi_{k+1},\ldots\}.
\end{equation}
If $W \subset V$ is any subspace of dimension $k$, then
\begin{equation}
W \cap V_{k+} \neq \{0\}.
\end{equation}
\end{lemma}
\begin{proof}
Let $\{\psi\}_{j=1}^k$ be a basis in $W$ and
let $Q:V\mapsto V_{k-}$ be the orthogonal projection on $V_{k-}$ where
\[
V_{k-} = \operatorname{span}\{\varphi_j\}_{j=1}^k
=\operatorname{span}\{\varphi_{1},\ldots,\varphi_{k}\}.
\]
Notice that $\dim V_{k-} = \dim W = k$ and that
\[
Q v = \sum_{j=1}^k(\varphi_j,v)\varphi_j.
\]
We consider two cases:
\begin{itemize}
  \item[\textsf{Case 1:}] There exists a $\psi \in W$, $\psi\neq 0$, such that
    $Q\psi = 0$;
  \item[\textsf{Case 2:}] For all $\psi\in W$, $\psi\neq 0$ we have
    $Q\psi\neq 0$.
\end{itemize}
In the first case, as $Q\psi=0$ we obtain that
\begin{eqnarray*}
&&W \ni \psi = (I-Q)\psi \in V_{k+}, \quad\mbox{and hence}\\
&&W\cap V_{k+} \supset \{\psi\}\neq \{0\},
\end{eqnarray*}
which shows the result of the lemma in the case when the null-space of $Q$ is non-trivial.

Consider now the second case, namely, $Q\psi\neq 0$ for all
$\psi\in W$. This implies that the matrix $C_{ij}=(\varphi_j,\psi_i)$,
$i,j=1,\ldots,k$ is nonsingular. Indeed, if this matrix is singular,
so is its transpose. Let $x\in \mathbb{R}^k$ be such that $C^Tx=0$,
i.e.
\[
C^Tx =0 = \sum_{j=1}^k (\varphi_i,\psi_j)x_j = 0, \quad i=1,\ldots,k,
\]
Thus, for $\psi = \sum_{j=1}^k x_j \psi_j$ we have
$(\varphi_i,\psi)=0$ for $i=1,\ldots,k$ and hence $Q \psi = 0$ where
$\psi \in W$ and $\psi \neq 0$.

Further, as $C$ is nonsingular, it follows that there exists $y\in \mathbb{R}^k$ such that
$C y = e_k$, $e_k = (0,\ldots,0,1)^T$. Therefore, with
\begin{equation}
  \psi = \sum_{j=1}^ky_j\psi_j,\label{e:wk}
\end{equation}
we have
\[
(\varphi_k,\psi) = 1, \quad\mbox{and}\quad
(\varphi_i,\psi) = 0,\quad i=1,\ldots,(k-1).
\]
Finally, these identities imply that
\begin{eqnarray*}
\psi & = & \sum_{i=1}^{\infty} (\varphi_i,\psi) \varphi_i\\
 & = &   \varphi_k + \sum_{i=1}^{k-1} (\varphi_i,\psi) \varphi_i
+\sum_{i=k+1}^{\infty} (\varphi_i,\psi) \varphi_i \\
& = & 
\sum_{i=k}^{\infty} (\varphi_i,\psi) \varphi_i.
\end{eqnarray*}
The right side of the identity above shows that $\psi\in V_{k+}$ and by~\eqref{e:wk}
we have $\psi \in W$ which completes the proof.
\end{proof}

We define the following set of $k$-dimensional subspaces:
\begin{equation}
\mathcal{W}_k = \left\{W\subset X \;\big|\; \dim W = k \right\}.
\end{equation}
The theorem which we prove next, shows also a min-max principle~\cite{weinstein1972methods,weinberger1974variational}. Its
proof is elementary, and relies on Lemma~\ref{l:intersection}.
\begin{theorem}[Min-max principle B]
  Let $V$ be a Hilbert space and $X\subset V$ is a dense subset of it.
  Let us assume that the eigenvectors $\{\varphi_j\}_{j=1}^{\infty}$
  of a selfadjoint operator $A: X \mapsto V$ form a complete
  orthonormal basis for $V$.  We then have the following min-max
  identity:
\begin{equation}
\lambda_k = \min_{W\in \mathcal{W}_k}\sup_{v\in W}\frac{(Av,v)}{(v,v)}
\end{equation}
\end{theorem}
\begin{proof}
Let $\{[\lambda_j,\varphi_j]\}_{j=1}^{\infty}$ be the eigenpairs of $A$ with
$\lambda_{1}\le\lambda_2\le\ldots\le\lambda_k\le\ldots$ and let
$W_{k-} = \operatorname{span}\{\varphi_k\}_{j=1}^{k}$.  We have that $W_{k-}\subset X$, because
$A\varphi_j = \lambda_j\varphi_j$, $j=1,\ldots,k$. Next, for any $v\in W_{k-}$ we obtain
\[
\frac{(Av,v)}{(v,v)} \le
\frac{\lambda_k(v,v)}{(v,v)}=\lambda_k\quad\Rightarrow\quad
\sup_{v\in W_{k-}}\frac{(Av,v)}{(v,v)} \le \lambda_k.
\]
Next, let $W\in \mathcal{W}_k$ be a space of dimension $k$ and we denote
$V_{k+} = \operatorname{span}\{\varphi_k\}_{j=k}^{\infty}$.
By Lemma~\ref{l:intersection} we have that there exists $\psi_W\neq 0$ such that
$\psi_W\in V_{k+} \cap W$. Hence,
\begin{eqnarray*}
\frac{(A\psi_W,\psi_W)}{(\psi_W,\psi_W)} & = &
\frac{\sum_{j=k}^{\infty}\lambda_j(\varphi_i,\psi_W)^2}{\sum_{j=k}^{\infty}(\varphi_i,\psi_W)^2}\\
 & \ge &
\frac{\lambda_k\sum_{j=k}^{\infty}(\varphi_i,\psi_W)^2}{\sum_{j=k}^{\infty}(\varphi_i,\psi_W)^2}
=\lambda_k.
\end{eqnarray*}
Taking the infimum over all spaces in $\mathcal{W}_k$ and we have:
\begin{eqnarray*}
\lambda_k & \le &
\min_{W\in \mathcal{W}_k}\frac{(A\psi_W,\psi_W)}{(\psi_W,\psi_W)}
 \le 
\min_{W\in \mathcal{W}_k}\sup_{v\in W}\frac{(A v,v)}{(v,v)}\\
& \le & \sup_{v\in W_{k-}}\frac{(A v,v)}{(v,v)}  \le  \lambda_k.
\end{eqnarray*}
which completes the proof.
\end{proof}

\section{Spectral properties of discretized elliptic operators}\label{s:algebraic-spectral}
We now discuss the spectral properties of the operator ${\mathcal L}$
given in \eqref{Model0}.  The following theorem, regarding the
eigenfunctions and eigenvalues of $\mathcal{L}$ is well-known
consequence of the Hilbert-Schmidt theorem for compact operators. Its
proof is found, for example, in~\cite{MR1892228,MR2987297}.

\begin{theorem}
  The operator ${\mathcal L}$ given in~\eqref{Model0} has a complete
  set of eigenfunctions $(\varphi_k)$ and nonnegative eigenvalues
$$
0\le\lambda_{1}\le\lambda_2\le \ldots
$$
such that
$$
{\mathcal L}\varphi_k=\lambda_k\varphi_k, \quad k=1, 2, 3\ldots.
$$
\begin{enumerate}
\item $\lim_{k\to\infty}\lambda_k=\infty.$
\item $(\varphi_i)$ forms an orthonormal basis of $V$ as well as for
  $L^2(\Omega)$.
\end{enumerate}
\end{theorem}
We also note that when $q=0$ and $\alpha(x)=I$ we have
\begin{enumerate}
\item For pure Neumann problem, $\lambda_{1}=0$ and $\varphi_{1}$ is
the  constant function.
\item For pure Dirichlet problem, $\lambda_{1}>0$ is simple and
  $\varphi_{1}$ does not change sign.
\end{enumerate}

For the case of Laplacian operator $\mathcal{L}=(-\Delta)$, we have
the well-known Weyl's estimate on the asymptotic behavior of its
eigenvalues, as shown
in~\cite{weyl1911asymptotische,wbyii1912asymptotische,reed1978iv}.
\begin{lemma}[Weyl's law]\label{lem:Weyl}
  Assume that $\Omega$ is contented. Then for the homogeneous Dirichlet
  problem,  the eigenvalues of the Laplacian operator satisfy:
\begin{equation}\label{Weyl0}
    \lim_{k\rightarrow\infty}\frac{\lambda_k}{k^{\frac{2}{d}}}=w_{\Omega},
    \mbox{ with }
    \quad w_{\Omega}=\frac{(2\pi)^2}{[\omega_d \operatorname{Vol}(\Omega)]^{\frac{2}{d}}},
\end{equation}
where $\omega_d$ is a volume of the unit ball in $\mathbb{R}^d$, and
the eigenvalues of the operator $\mathcal L$ given in \ref{Model0}
satisfy:
\begin{equation}\label{Weyl}
(\alpha_0w_{\Omega})k^{\frac2d}\le \lambda_k\le(\alpha_1 w_{\Omega})k^{\frac2d}, \quad \forall k\ge 1.
\end{equation}
\end{lemma}
We recall that, by definition, $\Omega\subset \mathbb{R}^d$ is a
\emph{contented domain} if it can be approximated as close as we
please by unions of $d$-dimensional cubes (see
\cite[p.~271]{reed1978iv} for the precise statement of such
definition).  Since finite element method is often used to discretize
problems on Lipschitz polyhedral domains, we note that
in~\cite{MR1158660} it was shown that all Lipschitz polyhedrons are
contented domains.

\begin{remark}
    We remark that by the Min-max principle (Theorem~\ref{thm:minmax}) the
    asymptotic behavior of $\{\lambda_k\}$ is of the same order with respect to $k$ when $q(x)\ge 0$ is bounded coefficient and a matrix valued
    (symmetric and positive definite) $\alpha(x)$. 
\end{remark}

In the following theorem, we extend Weyl's law stated in
Lemma~\ref{lem:Weyl} to the discretized Laplacian operator defined
in~\eqref{vph}.
\begin{theorem}\label{thm:WeylFE}
  Let $V_h\subset H_0^1({\Omega})$ be a family of finite element
  spaces on a quasi-uniform mesh with $\dim V_h = N$.  Consider the discretized operator of \eqref{Model0}
\[
    \mathcal L_h: V_h\mapsto V_h, \quad (\mathcal L_h u, v) = a(u, v), \quad \forall u, v\in V_h,
\]
and its eigenvalues:
\[
\lambda_{h,1}\le \lambda_{h,2}\le \cdots \le\lambda_{h,N}.
\]
Then, for all
  $1\le k\le N$, there exists a constant $C_w>0$ independent of $k$
  such that we have the following estimates:
    \begin{equation}\label{e:discrete-continuous}
        \lambda_k\le\lambda_{h,k}\le C_w\lambda_k.
    \end{equation}
and
    \begin{equation}\label{discrete-Weyl}
        \gamma_0 k^{2/d}\le\lambda_{h,k}\le \gamma_1k^{2/d}.
    \end{equation}
\end{theorem}
\begin{proof} For clarity, we present the proof for the Laplacian
  operator as the proof for the more general case is identical.  Using
  the infinite dimensional version of the min-max principle,
  Theorem~\ref{thm:minmax}, for a symmetric bilinear form
  $a(\cdot,\cdot): X\times X\mapsto \mathbb{R}$ with a dense domain
  $X\subset V$ and we have
\[
    \lambda_k = \inf_{\substack{W\subset X\\ \dim{W}=k}}\sup_{w\in W, w\ne 0}\frac{a(w,w)}{\|w\|_0^2}.
\]
As $V_h \subset X \subset V$ we have the inequality,
\begin{equation}\label{lambdakh}
  \begin{array}{rcl}
    \displaystyle  \lambda_{h,k} & = & \displaystyle
    \inf_{\substack{W\subset V_h\\ \dim{W}=k}}\sup_{w\in W, w\ne 0}\frac{a(w,w)}{\|w\|_0^2}\\
    & \ge & \displaystyle \inf_{\substack{W\subset X\\ \dim{W}=k}}\sup_{w\in W, w\ne 0}\frac{a(w,w)}{\|w\|_0^2} =\lambda_k,
\end{array}
  \end{equation}
because the infimum on the left is taken over a smaller collection of
spaces. This proves the lower bound in~\eqref{e:discrete-continuous}.

To show the upper bound, we first consider $k=N$. Since the finite
element mesh is quasi-uniform, by the inverse inequality, we have
    \[
        a(v,v)\lesssim h^{-2}\|v\|_0^2, \quad \forall v \in V_h.
    \]
    Therefore
    \[
        \lambda_{h,N}=\max_{v\in V_h}\frac{a(v,v)}{\|v\|_0^2}\lesssim h^{-2}
    \]
    Clearly, for all $k$ such that $\lambda_k\ge \frac{1}{2c_2h^2}$, where
    $c_2$ is the constant from \eqref{approximation}, we have
$$
    \lambda_{h, k}\le \lambda_{h, N}\lesssim h^{-2}\le 2c_2\lambda_k.
$$

    Next, we consider the case for $k$ such that
    \begin{equation}
      \label{small-lambdak}
\lambda_k<
    \frac{1}{2c_2h^2}.
    \end{equation}
 Let $W_k\subset V$ be the space spanned by the first $k$ eigenvectors of $-\Delta$, namely,
    \[
        W_k=\operatorname{span}\{\varphi_j\}_{j=1}^k.
    \]
Since
$\frac{|w|_{a}^2}{\|w\|_0^2}\le \lambda_k$ for all $w\in W_k$, from \eqref{approximation} we have
\begin{equation}\label{I-Qh}
  \begin{array}{rcl}
    \|(I-\Pi_h)w\|_0^2 & \le & c_2h^2|w|_{a}^2\\
    & \le & c_2h^2\lambda_k\|w\|_0^2\le \frac{1}{2}\|w\|_0^2.
 \end{array}
\end{equation}
This implies
\begin{equation}\label{e:lower-q}
  \begin{array}{rcl}
    \|\Pi_hw\|_0 & \ge &\|w\|_0-\|(I-\Pi_h)w\|_0 \\
    & \ge & \displaystyle \left(1-\frac{\sqrt{2}}{2}\right)\|w\|_0, \quad \forall w\in W_k.
    \end{array}
\end{equation}
    The above inequality implies that if $w\in W_k$ is such that $\Pi_hw=0$, then $w=0$.
This implies that $\{\Pi_h\varphi_j\}_{j=1}^k$ are linearly independent. We further denote
\[
W_{h, k}:=\Pi_hW_k=\operatorname{span}\{\Pi_h\varphi_j\}_{j=1}^k\subset V_h.
\]
We now use \eqref{stability}
and~\eqref{e:lower-q} and we have
    \begin{eqnarray*}
      \lambda_{h,k} &\le&
      \sup_{v=\Pi_hw\in W_{h,k}, v\neq 0}\frac{|v|_{a}^2}{\|v\|_0^2}
=
      \sup_{w\in W_{k}, w\neq 0}\frac{|\Pi_hw|_{a}^2}{\|\Pi_hw\|_0^2}\\
&\le&
        \sup_{w\in W_k, w\ne 0}\frac{c_1}{\left(1-\frac{\sqrt{2}}{2}\right)^2}\frac{|w|_{a}^2}{\|w\|_0^2}=
\frac{c_1}{\left(1-\frac{\sqrt{2}}{2}\right)^2}\lambda_k.
    \end{eqnarray*}
    This completes the proof.
\end{proof}

\section{An error estimate}\label{s:error}

In this section, we provide an error estimate for the finite element approximations of eigenvalue problems. Before we state the main result of this section let us comment on similar estimates found in the literature. The classical work~\cite{strang1973analysis} provides the estimate
 \begin{equation}\label{strangfix}
  0< \frac{\lambda_{k,h}-\lambda_k}{\lambda_k}\le h^{2D} \lambda_k^{D},
 \end{equation}
where $D$ is the polynomial degree of the finite element space. More recently, Zhang~\cite{zhang2015many} combines \eqref{strangfix} with Babu\v{s}ka-Osborn theory and shows that
 \begin{equation}\label{zhang2015}
  0< \frac{\lambda_{k,h}-\lambda_k}{\lambda_k}\approx h^{2D} \lambda_k^{D},
 \end{equation}
Furthermore, the results in~\cite{lin2014lower}, when combined with~\eqref{strangfix} and \eqref{zhang2015} imply that
 \begin{equation}
  0< \lambda_{k,h}-\lambda_k\le \sup_{\substack{u\in E_k\\\|u\|_V=1}}\inf_{v\in V_h}\|u-v\|_V^2.
 \end{equation}
 The main result of this section can be stated as the following theorem and provides an improved estimate for small $k$.
\begin{theorem}\label{thm:error-estimate}
For all $k$ satisfying \eqref{small-lambdak} with $\Pi_h=Q_h$, the following error
estimates hold:
\begin{equation}
  \label{eq:2}
  0\le \sqrt{\lambda_{k,h}}-\sqrt{\lambda_k}\le C \sup_{w\in W_k, \|w\|_0=1}|(I-Q_h)w|_{a}.
\end{equation}
    Here $C=1+\frac{c_1}{2}$ with $c_1$ being the constants in \eqref{stability}, which is independent of $k$ and $h$.
\end{theorem}
\begin{proof}
    If $k$ satisfies \eqref{small-lambdak},  by \eqref{I-Qh}
\begin{equation}\label{I-Q}
    \|(I-Q_h)w\|_0^2 \le \frac{1}{2}, \quad \forall w\in W_k, \|w\|_0=1.
\end{equation}
Since by~\eqref{e:lower-q} the image $Q_hW_k$ is a $k$-dimensional space,
it follows that
\begin{eqnarray}
  \lambda_{h,k} &\le& \sup_{w\in W_k, \|w\|_0=1}\frac{a(Q_hw,Q_hw)}{\|Q_hw\|_0^2}  \nonumber\\
  & = & \sup_{w\in W_k, \|w\|_0=1}\frac{a(Q_hw,Q_hw)}{1-\|(I-Q_h)w\|_0^2} \label{ineq4}\\
    & \le & \sup_{w\in W_k, \|w\|_0=1}a(Q_hw, Q_hw) (1+2\|(I-Q_h)w\|_0^2)\nonumber
\end{eqnarray}
Then
  \begin{eqnarray}
    \sqrt{\lambda_{h,k}} & \le &
    \sup_{w\in W_k}|Q_hw|_{a} \sqrt{1+2\|(I-Q_h)w\|_0^2}\nonumber\\
    & \le&  \sup_{w\in W_k}|Q_hw|_{a} (1+\|(I-Q_h)w\|_0^2)
    \label{ineq1}
\end{eqnarray}
Here we have used the following inequalities:
\[
    \sqrt{1+2x}\le 1+x\le \frac{1}{1-x}\le 1+2x, \quad x\in [0, 1/2].
\]
From \eqref{approximation}, we have
\begin{eqnarray*}
  \|(I-Q_h)w\|_0^2& = & \|(I-Q_h)(I-Q_h)w\|_0^2\\
  &\le& c_2h^2|(I-Q_h)w|_{a}^2.
\end{eqnarray*}
Combining with \eqref{I-Q}, we obtain
\begin{eqnarray}
    \|(I-Q_h)w\|_0^2
    & \le & \frac{1}{\sqrt{2}}\|(I-Q_h)w\|_0\nonumber \\
    & \le & \sqrt{\frac{c_2}{2}}h|(I-Q_h)w|_{a}.
    \label{ineq2}
\end{eqnarray}
Since for all $w\in W_k$, $\|w\|_0=1$, we have
\[
\lambda_k=\sup_{v\in W_k, \|v\|_0=1}a(v, v) \ge a(w,w),
\]
and hence we have that
\[
\sqrt{\lambda_k} \ge |w|_{a}.
\]
By \eqref{stability} and \eqref{small-lambdak} for all
$k\in W_k$, and $w$ with $\|w\|_0=1$ we obtain
\begin{equation}\label{ineq3}
    |Q_hw|_{a}\le c_1|w|_{a}\le c_1\sqrt{\lambda_k} \le \frac{c_1}{\sqrt{2c_2}h}.
\end{equation}
From \eqref{ineq1}, \eqref{ineq2} and \eqref{ineq3}, it then follows that
\[
    \sqrt{\lambda_{h,k}}\le\sup_{w\in W_k, \|w\|_0=1}\left(|Q_hw|_{a} + \frac{c_1}{2}|(I-Q_h)w|_{a}\right).
\]
This leads to
\begin{eqnarray*}
  \sqrt{\lambda_{h,k}}-\sqrt{\lambda_k} & \le & \sup_{w\in W_k, \|w\|_0=1}\left(|Q_hw|_{a} + \frac{c_1}{2}|(I-Q_h)w|_{a}\right)\\
  &&- \sup_{w\in W_k, \|w\|_0=1}|w|_{a} \\
  & \le & \sup_{w\in W_k, \|w\|_0=1}\left(|Q_hw|_{a} - |w|_{a}
  + \frac{c_1}{2}|(I-Q_h)w|_{a}\right)\\
    & \le & \sup_{w\in W_k, \|w\|_0=1}\left(|(I-Q_h)w|_{a} + \frac{c_1}{2}|(I-Q_h)w|_{a}\right)\\
    & = & \left(1+\frac{c_1}{2}\right)\sup_{w\in W_k, \|w\|_0=1}|(I-Q_h)w|_{a}.
\end{eqnarray*}
This completes the proof.
\end{proof}

\section*{Acknowledgments}
The work of Xu was partially supported by the DOE Grant
DE-SC0009249 as part of the  on Mathematics for Mesoscopic Modeling of Materials and by NSF grants DMS-1522615 and DMS-1819157.  The work of Zikatanov was partially supported by NSF grants DMS-1720114 and DMS-1819157.

\bibliographystyle{unsrt}
\bibliography{refs}
\end{document}